\documentclass[a4paper,10pt]{amsart}
\usepackage[utf8x]{inputenc}
\usepackage{amssymb,latexsym}
\usepackage{amsmath}
\usepackage{graphicx}

\newtheorem{theorem}{Theorem}[section]
\newtheorem{lemma}[theorem]{Lemma}
\newtheorem{assertion}[theorem]{Assertion}

\newtheorem{conjecture}[theorem]{Conjecture}
\newtheorem{corollary}[theorem]{Corollary}

\newtheorem{observation}[theorem]{Observation}

\theoremstyle{remark}

\newtheorem{example}[theorem]{Example}

\renewcommand{\geq}{\geqslant}
\renewcommand{\leq}{\leqslant}
\renewcommand{\ge}{\geqslant}
\renewcommand{\le}{\leqslant}

\def\cref#1{Corollary~$\ref{#1}$}

\def\Z{\mathbb{Z}}

\usepackage{array,color,colortbl}

\title{Covers in partitioned intersecting hypergraphs}

\author{Ron Aharoni}\thanks{The research of the first  author was
supported by  BSF grant no.
2006099, by an ISF grant  and by the Discount Bank
Chair at the Technion.}
\address{Department of Mathematics\\ Technion}
\email[Ron Aharoni]{raharoni@gmail.com}
\author{ C.J. Argue}
\address{Department of Mathematics, University of Chicago}
\email[C.J. Argue]{cjargue@gmail.com}

\begin{document}

\maketitle

\begin{abstract}
Given an integer $r$ and a vector $\vec{a}=(a_1, \ldots ,a_p)$ of
positive numbers with $\sum_{i \le p} a_i=r$, an $r$-uniform
hypergraph $H$ is said to be {\em $\vec{a}$-partitioned} if
$V(H)=\bigcup_{i \le p}V_i$, where the sets $V_i$ are disjoint, and
$|e \cap V_i|=a_i$ for all $e \in H,~~i \le p$. A $\vec{1}$-partitioned
hypergraph is said to be $r$-{\em partite}. Let $t(\vec{a})$ be the
maximum, over all intersecting $\vec{a}$-partitioned hypergraphs
$H$, of the minimal size of a cover of $H$. A famous conjecture of
Ryser  is that $t(\vec{1})\le r-1$. Tuza \cite{tuza}
conjectured that if $r>2$ then $t(\vec{a})=r$ for every two
components vector $\vec{a}=(a,b)$. We prove this conjecture whenever
$a\neq b$, and also for $\vec{a}=(2,2)$ and $\vec{a}=(4,4)$.

\end{abstract}

\section{Introduction}
An {\em $r$-graph} is an $r$-uniform hypergraph, namely a set of
sets called {\em edges}, all of size $r$. For a hypergraph $H$ we
denote by $\nu(H)$ the largest size of a matching (set of disjoint
edges) in $H$, and by $\tau(H)$ the minimal size of a cover (a set
of vertices meeting all edges). Since the union of all edges in a
maximal matching is a cover, in an $r$-graph $\tau \le r\nu$. Ryser
 conjectured that in $r$-partite hypregraphs $\tau \le
(r-1)\nu$. The conjecture appeared in a Ph.D thesis of his student, Henderson. At about the same time Lov\'asz \cite{lovasz} made the stronger conjecture
that in a non-empty $r$-partite hypergraph $H$ there exists a set $S$ of $r-1$ vertices such that $\nu(H-S) \le \nu(H)-1$.
 The case $r=2$ of Ryser's conjecture is K\"onig's theorem \cite{konig}. The
conjecture was proved for $r=3$ in \cite{ryser3}. The Lov\'asz version is open even in this case. The fractional
version, $\tau^* \le (r-1)\nu$, is a corollary of a famous theorem
of F\"uredi \cite{furedi}.

Of particular interest is the case $\nu=1$, namely that of
intersecting hypergraphs. As defined in the abstract, given a vector
$\vec{a}=(a_1, \ldots ,a_p)$ of positive numbers with $\sum_{i \le
p} a_i=r$, an $r$-uniform hypergraph $H$ is said to be {\em
$\vec{a}$-partitioned} if $V(H)=\bigcup_{i \le p}V_i$, where the
sets $V_i$ are disjoint, and $|e \cap V_i|=a_i$ for all $e \in H,~~i
\le p$. We denote by $t(\vec{a})$ the maximum of $\tau(H)$ over all
$\vec{a}$-partitioned intersecting hypergraphs. (Using a common abbreviated notation, we shall write $t(a_1, \ldots ,a_m)$ for  $t((a_1, \ldots ,a_m))$.) So, the $\nu=1$ case of Ryser's
conjecture is that $t(\vec{1})\le r-1$. This is known to be sharp
for $r$ for which there exists an $r$-uniform projective plane. In
\cite{abu, ronjanosian} the conjecture was shown to be sharp also for the first value
of $r$ for which an $r$-uniform projective plane does not exist, namely $r=7$. It
is plausible that the conjecture is sharp for all $r$.

A natural question is whether weaker conditions than $r$-partiteness
suffice to guarantee  $\tau \le r-1$. In the fractional case, this is indeed true: by F\"uredi's
theorem $\tau^* \le (r-1)\nu$  for all
$\vec{a}$-partitioned hypergraphs, for every non-trivial partitioning $\vec{a}$  of $r$. But it is likely that
$t(\vec{a})=r-1$ is quite rare. In \cite{tuza} the following was
suggested:

\begin{conjecture} \label{tuza}
If $r>2$ then $t(\vec{a})=r$ for all
two-components vectors $\vec{a}$. \end{conjecture}

Our main result is:

\begin{theorem}\label{main}
 If  $a\neq b$ then $t(a,b)=a+b$.
\end{theorem}

\section{Proof of Theorem \ref{main}}

Let $m \ge 2$.
A family $(H_1, H_2, \ldots, H_m)$ of $a$-graphs is called {\em cross intersecting} if  any two edges $e,f$
taken from distinct $H_i$s have nonempty intersection. If all $H_i$ are $a$-uniform, then this implies
\begin{enumerate}
\item

 $\tau(H_i)\le a$ (any edge from $H_j,~j \neq i$ is a cover for $H_i$)

 \item

$\tau\left(\bigcup_{i\leq m} H_i\right) \le 2a-1$ ( the union of two edges, taken from distinct $H_i$s,
is a cover for $\bigcup_{i \le m}H_i$.)

\end{enumerate}

 If $\tau(H_i)=a$ for all $i \le m$ and $\tau\left(\bigcup_{i\leq m} H_i\right) = 2a-1$ then the family is called {\em evasive}.

 \begin{observation}\label{largetau}
 If $(H_1, H_2, \ldots, H_m)$ (where $m\ge 2$) is an evasive system of cross intersecting $a$-graphs then $\tau(H_i)=a$ for all $i \le m$.

 \end{observation}
The reason - if there exists a cover $C$ of $H_i$ of size smaller than $a$, then for any edge $e \in H_i$  the set $C \cup e$ is a cover of size smaller than $2a-1$.

\begin{lemma}\label{evasive} \nopagebreak \hfill
\begin{enumerate}
\item For every $a$ there exists an evasive cross intersecting family of three $a$-graphs.
\item If $a=p^s$ for some prime $p$ and integer $s$ then there exists an evasive cross intersecting family of $a+1$ $a$-graphs.
\end{enumerate}

\end{lemma}

  \begin{proof}

       (1)~~ Consider the set of vertices
        $V=\{v_{ij}: 1\leq i,j\leq a\}$.  Let $H_1$ be the matching
        consisting of the rows of this grid, namely the edges of $H_1$ are $R_i:=\{v_{ij} \mid j \le a\}$, and let $H_2$
be the matching consisting of the columns of the grid, namely its
edges are $C_j:=\{v_{ij} \mid i \le a\}$. Let $H_3$ be the set of all
edges of the form
        $e_\sigma = \{v_{i,\sigma(i)}: 1\leq i \leq a\}$ where $\sigma$ is a permutation of
        $[a]$.

Clearly, $(H_1,H_2,H_3)$ is cross-intersecting. To show evasiveness,  consider a cover $C$ of $H_1\cup H_2\cup H_3$. Since $C$
covers $H_3$, by Hall's theorem there are sets $S,T\subseteq [a]$
such that         $M := \{v_{s,t}: s\in S, t\in T\} \subseteq C$ and
$|S|+|T|>a$. Let $d=\min(|S|,|T|)$. Since  $C$ covers $H_1\cup H_2$, we have $|C \setminus M| \ge a - d$.
         thus $|C| \geq |S|\cdot |T| +d =a(a-d)+(a-d)=a(a-d+1)$, and since $a-d+1\ge 2$ we have $|C| \geq 2a-1$, as required. \\

         (2)~~In this case there exists an $a+1$-uniform Desarguian projective plane (the projective plane ordinarily constructed, using vector spaces, is Desarguian). Take a line $v_1v_2\ldots v_{a+1}$ in this  plane, and let $H_i ~(i \le a+1)$ be the sets of edges containing $v_i$, with $v_i$ deleted. The system of $H_i$s thus obtained is called an {\em affine plane}, in this case - a Desarguian affine plane. Clearly, the system $(H_1, \ldots, H_{a+1})$ is cross intersecting. A well known result of Jamison, Brower and Schrijver \cite{brouwerschrijver,jamison} states that it is also evasive.
\end{proof}

\begin{lemma}\label{mainlemma}\nopagebreak \hfill

Let $F$ be any intersecting $b$-graph  with
$\tau(F)=b$.
 If there
exists an evasive cross intersecting family $(H_1,H_2, \ldots ,H_m)$
of $a$-graphs, and $\frac{a}{m-1}\leq b < a$, then there exists an $(a,b)$-partitioned $a+b$-graph $G$ in which
the $b$-side consists of $m$ disjoint copies of $F$, and $\tau(G)=a+b$.
\end{lemma}

\begin{proof}
         Let $F_1,\dots, F_m$ be $m$ disjoint copies of  $F$.
         Define $G = \{h\cup f: h\in H_i, f\in F_i, 1\leq i\leq m\}$. Clearly, $G$ is intersecting and is $(a,b)$-partitioned.
         We shall show that $\tau(G)=a+b$.  Let $C$ be a cover of $G$.
         If $C$ covers all $H_i$ then by the evasiveness property $|C| \ge 2a-1 \ge a+b$.
         For every $i \le m$, if $C$ does not cover  $H_i$ then it covers $F_i$. Hence, if $C$ does not cover any $H_i$ then $|C| \ge mb \ge a+b$.
         There remains the case that $C$ covers some $H_i$ but does not cover some $H_j$. In this
         case by Observation \ref{largetau} $|C \cap H_i|  \ge a$. Since also  $|C \cap F_j|  \ge b$,
         we have $|C| \ge a+b$.
\end{proof}

Note  that $F$ as in the lemma exists - for example  the collection of all $b$-subsets of $[2b-1]$. 
 Hence each part of the next lemma follows by combining Lemma \ref{mainlemma} with the corresponding part of Lemma \ref{evasive}:
\begin{lemma}\label{club}

Let  $b<a$ be integers. If either \begin{enumerate} \item $\frac{a}{2}\leq b$, or:

\item $a=p^s$ for some prime $p$ and integer $s$

\end{enumerate}
then  $t(a,b)=a+b$.

\end{lemma}


\begin{lemma}\label{bminusk}
If $k<b$ and both $(a,b)$ and $(b-k,k)$ satisfy either $(1)$ or $(2)$ then $t(a,b-k,k)=a+b$.
\end{lemma}

\begin{proof}
The assumption that $(b-k,k)$ satisfies either $(1)$ or $(2)$ implies by
Lemma \ref{club} that there exists an intersecting
$(b-k,k)$-partitioned $b$-graph $G$, with $\tau(G)=b$. Putting $F=G$
in Lemma \ref{mainlemma} yields the present lemma.
\end{proof}

    We now combine the results of Lemmas \ref{club} and \ref{bminusk} to prove Theorem \ref{main}.

    \begin{proof}
       As before, we shall assume that $b<a$. The only pair $(a,b)$ with $b<a\leq 8$ that does not satisfy either condition in
       Lemma \ref{club} is $(a,b)=(6,2)$.
       Since $(5,3)$ and $(2,1)$ both satisfy $(1)$ of Lemma \ref{club},  by Lemma \ref{bminusk}, $t(6,2)=t(5,2,1)=8$.

       Thus, we may assume that $a>8$ and $2b<a$. We proceed by induction on $a+b$.
Then there exists a pair of
numbers
$(u,v)$ satisfying \\ (i)~$\frac{u}{2}\leq v < u <a$~~\\ (ii)~ $u+v=a+b$  and\\
~~(iii)~$v\neq 2b$ (we shall use $v-b \neq b$).\\

Indeed, at least one of the pairs $(u,v)=\left( \left\lceil \frac{a+b+3}{2}
\right\rceil,\left\lfloor \frac{a+b-3}{2} \right\rfloor \right)$ and
            ~~ $~~ (u,v)=\left( \left\lceil \frac{a+b+1}{2} \right\rceil, \left\lfloor \frac{a+b-1}{2} \right\rfloor
            \right)$  satisfies all three conditions:
            the assumption that $a+b>9$ implies (i) for both pairs, and (iii) is satisfied by at least one of these pairs.
By the induction hypothesis $t(v-b,b)=v$, which by Lemma
\ref{bminusk} implies that $t(u,v-b,b)=a+b$, which in turn implies
that
            $t(u+v-b,b)=a+b$, namely $t(a,b)=a+b$.
    \end{proof}

\section{Other values of $\vec{a}$}

\subsection{An example showing
        $t(2,2)=4$.}

    \begin{example}
        Let $U=\{u_i: i\in \Z_5\}$, $W=\{w_i: i\in \Z_5\}$, and arrange each of $U$ and $W$ in a pentagon.
        We  construct a hypergraph $H$ with $10$ edges,  $\nu=1$ and $\tau=4$, as follows.
        Five edges, $e_1, \ldots ,e_5$,  consist each of an edge in the $U$-pentagon and a diagonal parallel to it in the $W$-pentagon
       (so, say, $e_i=\{u_i, u_{i+1}, w_{i-1},w_{i+2}\}$).
       The other five edges, $f_1, \ldots, f_5$, consist each of an edge of the $W$ pentagon, and the parallel diagonal of the $U$-pentagon, shifted by $1$ (so, say,
       $f_i=\{u_{i-1}, u_{i+1}, w_{i}, w_{i+1}\}$). Every two $e_i$s meet, because a diagonal in $W$ is disjoint only from the diagonals shifted $+1$ and $-1$ with respect to it, and then the corresponding parallel edges  in $W$ meet. The $f_i$s meet for a similar reason. Each $e_i$ meets each $f_j$ because of the shifting. Thus $\nu(H)=1$.
    \end{example}

       \begin{assertion} $\tau(H) = 4$.
       \end{assertion}

We have to show that any   $S\subset V$  of size $3$  is not a cover.
       If $S \subset U$ then it misses an edge of the complete graph on $U$, and hence by itself it is not a cover for $H$. Similarly, if $S \subset W$ then it is not a cover. Thus we may assume that (say)
        $|S\cap U|=2$ and $|S\cap W| = 1$.
        Let $E'$ be the set of edges that do not intersect $S\cap U$.
        There are distinct $i,j$ such that $\{e_i, e_j\}\subset E'$ or $\{f_i,f_j\}\subset E'$.

        For any $i,j$, $|e_i\cap e_j|=1$.
        By the choice of the edges $e_i,e_j$ this intersection point is in $U\backslash S$.
        Thus, $e_i\cap W$ and $e_j\cap W$ are disjoint, so $S$ cannot cover both $e_i$ and $e_j$.
        The argument is similar if $\{f_i,f_j\}\subset E'$.
        Thus $S$ is not a cover, which completes the proof that $\tau(H)=4$.

\subsection{An example showing  $t(2,2,2,2)=8$}

        Let $T$ be the  truncated $4$-uniform projective plane (namely, the projective plane of order $3$, with a vertex deleted). Then $T$ is a $3$-regular $4$-partite hypergraph, with $9$ edges and $12$ vertices partitioned into parts of size $3$.
         Denote its sides by $V_1,V_2,V_3,V_4$.
        For each $v\in V(T)$, let $G_v$ be a copy of $K_4$.
         For each $f=(f_1,f_2,f_3,f_4) \in E(T)$ and each $v=f_i \in F$
         choose arbitrarily a perfect matching $m(f,v)=\{m(f,v,0),m(f,v,1)\} $  in $G_v$. This can be done in such a way that for every $v \in V(T)$, denoting the three edges of $T$ going through $v$ by $f,g,k$,
         the three matchings $m(f,v), m(g,v), m(k,v)$ are distinct.

     Let $B$ be the set of all binary strings $\vec{\beta}=(\beta_1, \beta_2, \beta_3, \beta_4)$  such that $\beta_1=0$ if $\beta_3=\beta_4$ and $\beta_1=1$ if $\beta_3\neq \beta_4$. Clearly, $|B|=8$.
For every pair $(\beta, f) \in B \times E(T)$  let $h=h(\beta,f)$ be the set $\bigcup_{i \le 4} \bigcup m(f,f_i,\beta_i)$. So, $\vec{\beta}$ decides for each $i$ which of the two edges in the matching $m(f,f_i)$ is chosen by $h$.
Let $H$ be $\{h(\beta,f) \mid \beta \in B,~f \in E(T)\}$. Clearly, $H$ is $(2,2,2,2)$-partitioned, with sides
            $S_i = \bigcup_{v\in V_i} V(G_v)~~(i=1,2,3,4)$. It has $72$ edges.

\begin{assertion} $\nu(H) = 1$.
       \end{assertion}
\begin{proof}
For any $\vec{\beta},~\vec{\gamma} \in B$, if $\beta_3\neq \gamma_3$ and $\beta_4\neq \gamma_4$, then $\beta_1=\gamma_1$.
        Thus for any $f\in E(T)$ it is true that $h(f,\beta) \cap h(f,\gamma)\neq \emptyset$.
        Consider next two edges $h(\beta,f)$ and $h(\gamma,g)$ in $H$ for   $f \neq g$. By the intersection property of $P$,  there exists a vertex $v \in f \cap g$, say $v=f_i=g_i$. By the construction $m(f,v)\neq m(g,v)$. Hence the two edges in $S_i$ representing $f$ and $g$ in $h(\beta,f)$ and $h(\gamma,g)$ intersect.
                \end{proof}
\begin{assertion} $\tau(H) = 8$.
       \end{assertion}
\begin{proof} 
                 Assume for  contradiction that there exists a cover $C$ of size $7$.
	For each edge $e$, fix a vertex $x_e\in C$ contained in $e$. For each $e$, any pair $(y,e)$ where $y\in (C \cap e)\setminus \{x_e\}$ 
is said to constitute a {\em waste}. Since
                 $\sum_{v \in C}deg_H(v) = 7 \times 12 =84$ and $|E(H)|=72$,  we are allowed at most  $84-72=12$
                 wastes.

                Note that the common degree of two vertices belonging to different sides $S_i$ is  $2$,
                while the common degree of two vertices belonging to  the same $G_v$ is $4$.
                Suppose first that
                there exist two vertices in $C$ that meet a side in distinct
                $G_v$s. Then all other vertices, apart from possibly
                one that may belong to the
                third copy of $G_v$ in the same side, contributes  at
                least
                $4$ wastes, and thus the total number of wastes is more than $12$.
                Thus we may assume that each side of $H$ contains at
                most one $G_v$ containing points from $C$. Fixing a
                vertex $x$, every other vertex contributes with $x$
                 at least $2$
                wastes, while two vertices belonging to the
                same $G_v$ (which must exist) contribute $4$
                wastes, so again there are more than $12$ wastes.
    \end{proof}

\begin{corollary} $t(4,4)=8$. \end{corollary}

The first value of the function $t$ that is not settled is $t(3,3)$.

{\bf Acknowledgement} We are grateful to Aart Blokhuis for useful information.

\end{document}